\documentclass[10pt, reqno]{amsart}
\usepackage{amsmath, amsthm, amscd, amsfonts, amssymb, graphicx, color}
\usepackage[bookmarksnumbered, colorlinks, plainpages]{hyperref}

\makeatletter \oddsidemargin.9375in \evensidemargin \oddsidemargin
\newtheorem{theorem}{Theorem}[section]
\newtheorem{lemma}[theorem]{Lemma}
\newtheorem{proposition}[theorem]{Proposition}
\newtheorem{corollary}[theorem]{Corollary}
\theoremstyle{definition}
\newtheorem{definition}[theorem]{Definition}

\theoremstyle{remark}
\newtheorem{remark}[theorem]{Remark}
\numberwithin{equation}{section}

\begin{document}
\setcounter{page}{1}

\title[$N$-strongly quasi-invariant measure on  double coset]{$N$-strongly quasi-invariant measure on  double coset spaces}

\author[  F. Fahimian and R.A. Kamyabi-Gol, F. Esmaeelzadeh]{ F. Fahimian $^1$  and R.A. Kamyabi-Gol$^{2{*}}$ F. Esmaeelzadeh $^3$}

\address{$^{3}$ Department of  Mathematics, Bojnourd Branch, Islamic Azad university, Bojnourd, Iran.}
\email{\textcolor[rgb]{0.00,0.00,0.84}{esmaeelzadeh@bojnourdiau.ac.ir}}

\address{$^{1,2}$ Department of Mathematics, Center of Excellecy in Analysis on Algebric Structures (CEAAS), Ferdowsi University of Mashhad, Mashhad, Iran.}
\email{\textcolor[rgb]{0.00,0.00,0.84}{fatemefahimian@gmail.com}}
\email{\textcolor[rgb]{0.00,0.00,0.84}{kamyabi@um.ac.ir}}




\date{Received: xxxxxx; Revised: yyyyyy; Accepted: zzzzzz.
\newline \indent $^{*}$ Corresponding author}

\begin{abstract}
Let $G$ be a locally compact group, $H$ and $K$ be two closed subgroups of $G$, and $N$ be the normalizer group of $K$ in $G$. In this paper, the existence and properties of a rho-function for the triple $(K, G, H)$ and an $N$-strongly quasi-invariant measure of double coset space $K \backslash G /H$ is investigated. In particular, it is shown that any such measure arises from a rho-function. Furthermore, the conditions under which an $N$-strongly quasi-invariant measure arises from a rho-function are studied. 
\end{abstract} \maketitle

\section{Introduction}
Let $G$ be a locally compact group and $H$ and $K$ be closed subgroups of $G$. The double coset space of $G$ by $H$ and $K$ respectively, is 
\[
K \backslash G /H  =\{Kx H ;\ x \in G\}, 
\]
which induced by Liu in \cite{5}.\\
When $K$ is trivial, a double coset $K \backslash G /H$ changes to a homogeneous space $G/H$. The existence of quasi-invariant measures on homogeneous spaces $G/H$ (with merely measurable  rho-functions) was first proved by Mackey \cite{khkh6} under the assumption that $G$ is second countable. Bruhat \cite{khkh8} and Loomis \cite{khkh7} showed  how to obtain strongly quasi-invariant measures with no countability hypotheses. This work is extended in a special case in \cite{khkh90}. Also, the existence of a homomorphism rho-function causes  the existence of a relatively invariant measure on $G/H$ is in \cite{khkh999}. One may refer to \cite{khkh999,3} to find more informations about homogeneous space $G/H$. \\
When $K=H$, a double coset space $K \backslash G /H$ changes to a hypergroup in which the homogeneous space $G/H$ is a semi hypergroup \cite{kh8}. It is worthwhile to note that the hypergroup plays important rules in physics.\\
In this paper, we construct an $N$-strongly quasi-invariant measure on  $K \backslash G /H$ when $H$ and $K$ are used subgroups, not necessarily compact. Also we investigate when $K$ is a normal closed subgroup of $G$ then $K \backslash G /H$ possesses a $G$-strongly quasi-invariant measure. In addition, when $H$ is trivial we show  the existence of an $N$-strongly quasi-invariant measure on the right cosets of $K$ in $G$. \\
It is worth mentioning that in \cite{5} the conditions for the existence of $N$-relatively invariant measures and $N$-invariant measures are investigated. \\
Some preliminaries and notations about coset space $K \backslash G /H$ and related measures on it are stated in Section. \ref{sekh2}. \\
In Section. \ref{sekh3}, we construct a rho-function for the triple $(K, G, H)$ and introduce an $N$-strongly quasi-invariant measure which arises from this rho-function. \\
In particular, we obtain in Section. \ref{sekh4}., conditions under which an $N$-strongly quasi-invariant measure arises from a rho-function.
\section{Notations and preliminary results} \label{sekh2} 
Let $G$ be a locally compact Hausdorff group  and let $H$ and $K$ be closed subgroups of $G$. Throughout this paper, we denote the left Haar measures on $G$, $H$ and $K$ respectively, by $dx$, $dh$, $dk$, and their modular functions by $\Delta_G$, $\Delta_H$ and $\Delta_K$, respectively.  If $S$ is a locally compact Hausdorff space, a (left) action of $G$ on $S$ is a continuous map $(x,s)\mapsto xs$ from $G\times S$ to $S$ such that (i) 
$s \to xs$ is a homeomorphism of $S$ for each $x \in G$, and (ii)  
$x(ys)=(xy)s$ for all $x,y \in G$ and $s\in S$. 
A space $S$ equipped with an action of $G$ is called  a $G$-space. A $G$-space $S$ is called transitive if for every $s,t\in S$ there exists $x \in G$ such that $xs =t$. \\
The standard examples of transitive $G$-spaces are the quotient spaces $G/H$ (where $H$ is a closed subgroup of $G$), equipped with the quotient topology  on which $G$ acts by left multiplication. We shall use the term homogeneous space to mean a transitive space $S$ that is isomorphic to a quotient space $G/H$. 
 In   homogeneous space  $G/H$, if  $\mu$ is a positive Radon measure on $G/H$, Borel set $E$ is called negligible with respect to $\mu$, if $\mu(E)=0$. Let $\mu_x$ denote its transfer by $x \in G$, that is $\mu_x(E)=\mu(x\cdot E)$ for any Borel set $E \subseteq G/ H$. $\mu$  is called strongly quasi invariant if there is a positive continuous function $\lambda$ on $G \times G/H$ such that $d\mu _x (yH)=\lambda (x, yH)d\mu (yH)$, for all $x, y \in G$.  A rho-function for the pair $(G, H)$ is defined to be a positive locally integrable function $\rho$ on $G$ which satisfies
\[
\rho
(xh)=\frac{\Delta _H(h)}{\Delta_G(h)} \rho(x), \quad (x\in G, \ h \in H).
\]
It is  known that for each pair $(G, H)$ there is a strictly positive rho-function which constructs a strongly quasi-invariant measure $\mu$ on $G/H$ such that
\begin{equation}
\int _G f(x) \rho(x) dx =\int _{\frac{G}{H}} \int _H f(xh) dh d\mu (xH), \label{eqkh1.1}
\end{equation}
for all $f \in C_c(G)$, the space of all continuous functions on $G$ with compact supports. \\
And conversely, each strongly quasi-invariant measure on $G/H$ arises from a rho-function which satisfies \eqref{eqkh1.1} for a rho-function $\rho$, and all such measures are strongly equivalent. That is to say, all strongly quasi-invariant measures on $G/H$ have the same negligible sets (see \cite{3,khkh999}).\\
The notion of double coset space is a natural generalization of that of coset space arising by two subgroups, simultaneously. Recall that if $K \backslash G /H$ is a double coset space of $G$ by $H$ and $K$, then elements of $K \backslash G /H$ are given by $\{ Kx H ;\ x \in G\}$.\\
The canonical mapping of which, is $q: G \to K \backslash G /H$, defined by $q(x)=KxH$, which is abbreviated by $\ddot{x}$, and which is surjective. The double coset space $K \backslash G /H$ equipped with the quotient topology, which is the largest topology,  that makes $q$ continuous. In this topology $q$ is also an open mapping and {\bf proper}--that is for each compact set  $F \subseteq K \backslash G /H$ there is a compact set  $E \subseteq G$ with $q(E)=F$. Based on the above mentioned case,   $K \backslash G /H$ is a locally compact and Hausdorff space. \\
Let $N$ be the normalizer of $K$ in $G$, i.e., 
\[
N=\{g\in G; \ gK=Kg \}.
\]
Then, there is a naturally defined mapping
\[
\varphi:N \times K \backslash G /H \to K \backslash G /H
\]
given by
\[
\varphi\big(n, q(x)\big):=Knx H.
\]
It can be verified that $\varphi$ is a well-defined, continuous, transitive action of $N$ on $K \backslash G /H$. Considering $K \backslash G /H$ with this transitive action, we now  denote $\varphi\big(n, q(x)\big)$ by $n\cdot q(x)$.\\
We define the mapping $Q$ from  $C_c(G)$ to $C_c(K \backslash G /H)$ by 
\[
Q(f)(KxH)=\int_K \int_H f(k^{-1}xh)dh dk.
\]
It is evident that $Q$ is a well-defined continuous linear map, as well as $supp\big(Q(f)\big)\subseteq q(supp \ f)$. 
In the following,  the properties of this mapping is investigated. However,  we first  recall that the definition of $IN$-group  and verification of a property of it  is used in the sequel.\\
A locally compact group $G$ is called an $IN$-group if there is a compact unit neighbourhood $U$ in  $G$ which is invariant under inner automorphism, that is, for any $x \in G$, $x Ux^{-1}=U$. It is known that the $IN$-groups are unimodular.
\begin{lemma}\label{le2.2}
If $K$ is also  an $I N$-group, then $\int_Kf(k)dk=\int_K f(n k n^{-1})dk$, for all $f\in C_c(K)$ and $n \in N$. 
\end{lemma}
\begin{proof}
Let for  $n \in N$,  $\lambda_n : C_c(K) \to \mathbb{C}$ be given by 
\[
\lambda _n (f) =\int_K f(n kn^{-1}) dk.
\]
Then for every  $t\in K$, we have
\[
\lambda _n (L_{t^{-1}}f) =\int_K L_{t^{-1}} f(n kn^{-1})dk =\int_K f( tnkn^{-1})dk=\int_K f(nkn^{-1})dk. 
\]
This shows that  $\lambda _n$ is  left invariant, so it induced a left Haar measure $\lambda_n$ on $K$.
Therefore, there is $c>0$ such that 
\[
\int_K f(k)d\lambda _n (k) =c \int_K f(k) dk.
\]
Since $K$ is an $IN$-group, then there is a compact unit neighbourhood $U$ in $K$ such that $x Ux^{-1}=U$ for all $x \in K$. Thus,  $|n^{-1} Un|\leq|U|$ for each $n \in N$, where $|U|$ denotes the measure of $U$.  Therefore, we can write
\begin{align*}
|c-1| |U|&= \big||c|U|-|U|\big|\\
&=\big|\lambda _n(U) -|U|\big|=\big||nUn^{-1}|-|U|\big|\\
&\leq \big||U| -|U|\big|=0.
\end{align*}
This implies that 
 $c=1$. 
\end{proof}
\begin{lemma}\label{lekh2.3}
For any compact set  $F \subseteq K \backslash G / H$  there exists $f\in C_c^+(G)$ such that $Qf=1$ on $F$. 
\end{lemma}
\begin{proof}
The proof is straightforward. 
\end{proof}
Note that for $f\in C_c(G)$ and $g\in G$, we consider $L_g f(x)=f(g^{-1}x)$ and $R_gf(x)=f(xg)$, and for each $n \in N$ and $F\in C_c(K \backslash G /H)$, we define $L_gF(\ddot{x})=F\big((g^{-1}x)^{\ddot{\ }}\big)$ and $R_gf(\ddot{x})=F\big((xg)^{\ddot{\ }}\big)$.
\begin{lemma}\label{le2.3}
Given the notation at the beginning of the section,
the map $Q: C_c(G)\to C_c(K \backslash G /H)$  has the following properties .
\begin{enumerate}
\item[(i)]
$Q\big(C_c(G)\big)=C_c( K \backslash G/H)$
\item[(ii)]
If $K$ is also an $IN$-group, then for each $n \in N$ 
\[
Q(L_n f)= L_n Q(f) , \quad f\in C_c(G).
\]
\end{enumerate}
\end{lemma}
\begin{proof}
For (i) suppose that  $F\in C_c(K \backslash G /H) $. Since $q$ is proper, then  there is a compact subset $D\subseteq  G $ such that $q(D)=supp \ F$. Let $f\in C_c(G)$ be such that $f(d)>0$ for all $d\in D$. Consider the   function $f_1$  defined  on $G$ by 
\[
f_1(x)=\begin{cases}
\frac{f(x)\cdot F(q(x))}{Q(f)(q(x))} & \text{if} \quad Q(f)\big(q(x)\big) \neq 0 \\
0 & \text{if} \quad Q(f) \big(q(x)\big)=0 
\end{cases}.
\]
Since $Q(f)\big(q(x)\big)>0$ for $x \in q^{-1} \big((supp\ F)\big)$, and $F\big(q(x)\big)=0$, for $x \in G \setminus q^{-1}(supp \ F)$, which is an open subset in $G$,  $f_1 \in C_c(G)$ and $Q(f_1)=F$. \\
Finally, for  (ii) according to Lemma \ref{le2.2}, we may state that 
\begin{align*}
Q(L_n f)(KxH)&= \int_K \int_H f(nk^{-1}xh) dh dk \\
&=\int_H\int_K f(nk^{-1}n^{-1}nxh)dkdh\\
&=\int_H\int_K f(k^{-1}nxh)dkdh\\
&=L_n Q(f) (KxH) .
\end{align*}
\end{proof}
Next theorem gives a necessary and sufficient condition for the existence of a positive Radon measure on $K \backslash G /H$. 
\begin{theorem}\label{th2.4}
If $\mu$ is a positive Radon measure on $K \backslash G /H$, then positive Radon measure $\tilde{\mu}$ on $G$ is defined by 
\begin{equation}\label{eq2.1}
\int_G f(x) d \tilde{\mu}(x)=\int_{K \backslash G /H} Q(f) (\ddot{x})d\mu (\ddot{x}),
\end{equation}
satisfying
\begin{equation}\label{eq2.2}
\int_G f(k x h^{-1}) d\tilde{\mu}(x)=\Delta _K (k) \Delta_H(h) \int_G f(x) d\tilde{\mu}(x). 
\end{equation}
Conversely, if a positive Radon measure $\tilde{\mu}$ on $G$ has the property \eqref{eq2.2}, then the equation \eqref{eq2.1} defines a positive Radon measure $\mu$ on $K \backslash G /H$. 
\end{theorem}
\begin{proof}
Suppose that $\mu$ is a positive Radon measure on $K \backslash G /H$, then  $\tilde{\mu}$ defined by \eqref{eq2.1} is clearly a positive Radon measure on $G$. Also, for each $h_0 \in H$, $k_0\in K$, $f \in C_c(G)$, we have 
\begin{align*}
\int_G f(k_0 x h_0^{-1})d\tilde{\mu}(x)& =\int_{K \backslash G /H}\int_K\int_H L_{k_0^{-1}}\circ R_{h_0^{-1}} f(k^{-1}xh)dhdk d\tilde{\mu}(x)\cr
&=\int_{K \backslash G /H}\int_K\int_H f(k_0k^{-1}xhh_0^{-1})dhdk d\mu(\ddot{x})\\
&= \Delta_H (h_0)\Delta_K(k_0)\int_{K \backslash G /H} Q(f)(\ddot{x}) d\mu(\ddot{x}) \\
&=\Delta_H(h_0)\Delta_K(k_0) \int_G f(x) d\tilde{\mu}(x). 
\end{align*}
Conversely, suppose that the positive Radon measure $\tilde{\mu}$ on $G$ has the property \ref{eq2.2}, then  take
\[
\mu: C_c(K \backslash G /H)\to (0, +\infty),
\]
 by
\[
\mu\big(Q(f)\big)=\int_G f(x)d\tilde{\mu}.
\]
Now we show that $\mu$ is well-defined, let $f\in C_c(G)$ such that $Q(f)=0$. According to Lemma \ref{lekh2.3}, there is $g $ in $C_c(G)$ such that $Q(g)\equiv 1$ on $Q(supp \ f)$. By  using the Fubini's Theorem, we have 
\begin{align*}
\int_G f(x)d\tilde{\mu}&=\int_G f(x)Q\big(g(q(x)\big) d\tilde{\mu}(x)\\
&=\int_G f(x) \int_K \int_H g (k^{-1}xh) dh dk d\tilde{\mu}(x) \\
&=\int_K\int_H\int_G f(kxh^{-1})g(x) \Delta_K(k)\Delta_H(h)d\tilde{\mu}(x)dhdk \\
&=\int_G g(x) \int_H \int_K f(k^{-1}xh)  dkdhd\tilde{\mu}(x)\\
&=\int_G g(x) Q(f) \big(q(x)\big) d\tilde{\mu}(x)=0.
\end{align*}
It is easy to check that $\mu$ is a positive linear functional, therefore it induces a positive Radon measure $\mu$ on $K \backslash G /H$ such that,
\[
\int_G f(x) d \tilde{\mu}(x)=\int_{K \backslash G /H} Q(f) (\ddot{x})d\mu (\ddot{x}).
\]
\end{proof}
\begin{corollary}
Considering the assumptions of  Theorem \ref{th2.4},  there is a correspondence between the positive Radon measure $\tilde{\mu}$ on $G$ and $\mu$ on the double coset space $G//H$, such that
\[
\int_G f(x)d\tilde{\mu}(x)=\int_{G //H} Q(f) (\ddot{x}) d\mu(\ddot{x})
\]
and
\[
\int_G f(hxh^{-1})d\tilde{\mu}(x)=\int_G f(x)d\tilde{\mu}(x)
\]
for all $f \in C_c(G)$.
\end{corollary}
\section{The existence of $N$-strongly quasi invariant measure }\label{sekh3}
In this section, we refine and generalize the concept of strongly quasi invariant measure on double coset spaces. Moreover, we investigate the existence of $N$-strongly quasi-invariant measure on these spaces. We start our work with the following definitions.
\begin{definition}
Let  $G$ be a locally compact group and $H$ and $K$ be closed subgroups of it.
For a positive Radon measure $\mu$ on $K \backslash G / H$, assume that $\mu_n$ is its transfer by $n \in N$, that is, $\mu_n(E)=\mu(n\cdot E)$, for any Borel set $E\subseteq K \backslash G /H$. $\mu$ is called 
 $N$-strongly quasi invariant if there is a continuous positive function $\lambda$ on $N \times K \backslash G / H$ such that for all $n \in N$, $d\mu_n(\ddot{y})=\lambda(n,\ddot{y})d\mu(\ddot{y})\ (\ddot{y}\in K \backslash G / H)$.  We call such  $\lambda$ the modular function of $\mu$.
 \end{definition}
\begin{remark}
Note that if $K$ is normal in $G$, then the $N$-strongly quasi-invariant  measure $\mu$ is the $G$-strongly quasi invariant   on $K \backslash G / H$ and if $K=\{e\}$, $ \mu$ is the strongly quasi invariant measure on $G/H$.
\end{remark}
\begin{definition} \label{de2.1}
Suppose that $G$ is a locally compact group and $H$ and $K$ are closed subgroups of $G$. A rho-function for the triple $(K ,G , H)$ is a non-negative locally integrable function $\rho$ on $G$, which satisfies
\begin{equation*}
\rho (k x h) = \frac{\Delta_K (k) \Delta_H (h)}{\Delta_G (h)} \rho (x).
\end{equation*}
\end{definition}
In the following, it is shown that for every triple ($K, G, H)$ there exists a rho-function and an $N$-strongly quasi-invariant measure on $K \backslash G / H$, which arises from this rho-function. For this, first  it is shown that for each $f\in C_c(G)$ there exists a rho-function $\rho_f$ for the triple $(K, G, H)$. 
\begin{proposition} \label{pr2.2}
Suppose that $G$ is  a locally compact group and $H$ and $K$ are closed subgroups of $G$. Then for each $f \in C_c(G)$ there exists a continuous rho-function $\rho_f$ on $G$.
\end{proposition}
\begin{proof}
For each $f \in C_c (G)$,  take  
\begin{equation*}
\rho_f (x) = \int_K \int_H \frac{\Delta_G (h)}{\Delta_H(h) \Delta_K(k^{-1})} f (k^{-1} x h) d h d k.
\end{equation*}
It is clear that 
 $\rho_f$ is a  well-defined positive linear map and according to  Fubini's formula we have 
\begin{align*}
 \int_K \int_H \frac{\Delta_G (h)}{\Delta_H(h) \Delta_K(k^{-1})} f(k^{-1} x h) d h d k&= \int_H \int_K \frac{\Delta_G (h)}{\Delta_H(h) \Delta_K(k^{-1})} f(k^{-1} x h) d k d h\\
& = \int_{K \times H} \frac{\Delta_G (h)}{\Delta_H(h) \Delta_K(k^{-1})} f(k^{-1} x h) d (k \times h).
\end{align*}
First, we show that
  $\rho_f$ is uniformly continuous. Suppose that   $V$ is  a compact unit neighbourhood  in $G$.  Since $f\in C_c(G)$,  for given $\varepsilon>0$ there is a  symmetric neighbourhood $U$ of $e$ such that $U \subseteq V$ and for each $y\in Ux$, $|f(x)-f(y)|<\varepsilon$.\\
Take $M=V\cdot supp f \cdot V$. If $x \in G \backslash KMH$, then $f(k^{-1}xh)=f(k^{-1}yh)=0$ for all $k\in K$ and $ h \in H$,
and if $x \in KMH$,  there is $k_0\in K$ and $ h_0 \in H$ such that $k_0^{-1}xh_0 \in M$. If $y\in Ux $, then we have two cases:
\begin{enumerate}
\item[(1)]
If $k^{-1}k_0^{-1}y\in supp f$, then $k^{-1}k_0^{-1}x \in M$. Therefore, $k^{-1}\in M h_0M^{-1}\cap K$. Also, if $k^{-1}k^{-1}_0 x \in supp \ f$, then $k^{-1}\in Mh_0 M^{-1}\cap K $.
\item[(2)]
If $yh_0h\in supp\ f$, then $h \in M^{-1} k_0^{-1}M \cap H$.
Also, if
$xh_0h\in supp\ f\subseteq M$, then $h\in M^{-1} k_0^{-1}M \cap H $.\\
Now put $L=L_1\times L_2$, where $L_1=Mh_0 M^{-1}\cap K $ and $L_2=M^{-1}k_0^{-1} M \cap H$.  The set $L$ is  compact  in $K \times H$ 
and if $(k,h)\not\in L$  we have 
\[
f(k^{-1}k_0^{-1}xh_0h)=f(k^{-1}k_0^{-1}y h_0 h)=0. 
\]
Hence, according to the above mentioned, we can write
\begin{align*}
|\rho_f(x) -\rho_f(y)|&\leq \int_K\int_H \frac{\Delta_G(h)}{\Delta_H(h)\Delta
_K(k^{-1})}|f(k^{-1}xh) -f(k^{-1}yh) |dhdk \\
&=\int_K \int_H \frac{\Delta_G(h'h)}{\Delta_H(h'h)\Delta_K(k^{-1})}|f(k^{-1}xh'h) -f(k^{-1}yh' h)| dh dk \\
&= \int_K \int_H \frac{\Delta_G(h'h)}{\Delta
_H(h'h)\Delta_K((k'k)^{-1})} |f(k^{-1}k'^{-1}xhh)-f(k^{-1}k'^{-1}yh'h)| \\
&=\frac{\Delta_G(h')}{\Delta_H(h')\Delta_K(k')}\quad \int_{L_1}\int_{L_2} |f(k^{-1}k'^{-1}xh'h)-f(k^{-1}k'yh'h)| dh dk \\
&<   d\cdot d(h \times k)(L) \cdot \varepsilon
\end{align*}
where $\delta(h', k')$  denotes $\frac{\Delta_G(h')}{\Delta_H(h')\Delta_K(k')}$ and $d$ is $\max \{\delta(h', k'), (h', k') \in L\}$.
\end{enumerate}
Next, if $k_1 \in K$ and $h_1 \in H$ are arbitrary, we have 
\begin{align*}
\rho_f (k_1 x h_1) & = \int_K \int_H \frac{\Delta_G (h)}{\Delta_H(h) \Delta_K(k^{-1})} f(k^{-1} k_1 x h_1 h) d h d k \\
& = \int_K \int_H \frac{\Delta_G (h_1^{-1} h)}{\Delta_H(h_1^{-1} h) \Delta_K(k^{-1})} f(k^{-1} k_1 x h) d h d k\\
& = \frac{\Delta_H(h_1)}{\Delta_G(h_1)} \int_H \int_K \frac{\Delta_G (h)}{\Delta_H(h) \Delta_K(k^{-1} k_1^{-1})} f(k^{-1} x h) dk  dh  \\
& = \frac{\Delta_H(h_1) \Delta_K(k_1)}{\Delta_G(h_1)} \int_K \int_H \frac{\Delta_G (h)}{\Delta_H(h) \Delta_K(k^{-1})} f ( k^{-1} x h) d h d k \\
& = \frac{\Delta_H (h_1) \Delta_K (k_1)}{\Delta_G (h_1)} \rho_f (x).
\end{align*}
This proves $\rho_f$ is a rho-function.
\end{proof}

In the next proposition, we  will prove that for the triple $(K, G, H)$ there is a positive continuous rho-function whose support is $G$. But first we need the following technical Lemma. 
\begin{lemma} \label{le2.5}
Let  $U$ be a symmetric unit neighbourhood  of $ G$ with compact closure and $N$ be the normalizer of $K$ in $G$. If  $U_N =U \cap N$ is taken,  there exists a subset $A$ of $G$ with the following properties:
\begin{itemize}
\item[(i)]
For every $x \in G$, we have $K x H \bigcap U_N a \neq \varnothing$, for some $a \in A$.
\item[(ii)]
If $M $ is a compact subset of $G$, then $\{ a \in A ;K M H \bigcap U_N a \neq \varnothing \}$ is finite.
\end{itemize}
\end{lemma}
\begin{proof}
Let $\mathcal{A} = \{ A \subseteq G ;\ \text{for all}\ a\neq b \ \text{in}\ A, a \notin KU_NbH \}.$ According to Zorn's Lemma, $\mathcal{A}$ has a maximal element, say $A$. We claim that $A$ satisfies $(i)$ and $(ii)$.
\begin{itemize}
\item[(i)]
If $x \in A$  the claim is clear.
If $x \in G \setminus A$ where such that $K x H \bigcap U_N a = \varnothing$, for all $a \in A$, then we could add $x$ to $A$ and make $A$ strictly larger. So (i) holds for $A$. 
\item[(ii)]
Let $M $ be a  compact subset of $G$ and $A_M =\{ a \in A ; K M H \cap U_N a \neq \varnothing \}$.\\
For every $a \in A_M , K M H \cap U_N a \neq \varnothing$ implies $K a H \cap U_N M \neq \varnothing$ and conversely. \\
Pick $x_a\in KaH\cap U_N M$. If $A_M$ is infinite, then $\{x_a ; a \in A_M \}$  would have a cluster point $x$, say, in the compact set $\bar{U}_N M$.\\
Let $V$ be a unit neighbourhood  such that $V V^{-1} \subseteq U$. Then, by choosing $V_N=V\cap N$, we have $V_NV_N^{-1}\subseteq U_N$. Since the  $x_a$ is  a cluster at $x$,  there exist distinct $a , b \in A_M$ such  that $x_a  , x_b \in V_{N}x$. This implies that $x_a x_b^{-1} \in V_N V^{-1}_N \subseteq U_N$. But $x_a \in K a H$ and $x_b \in K b H$, so $x_a \in K U_N b H$ which forces $a \in K U_N b H$, in contradiction to $A \in \mathcal{A}$.  So, $A_M$ is finite and $(ii)$ is met.
\end{itemize}
\end{proof}
Next we use Lemma \ref{le2.5} and Proposition \ref{pr2.2} to give a rho-function for each  triple $(K, G, H)$ mentioned above,  which is strictly positive on $G$. 
\begin{proposition} \label{le2.6}
With the above notation,
 there exists a rho-function $\rho$ for the triple $(K , G , H)$, which is continuous and everywhere strictly positive on G.
\end{proposition}
\begin{proof}
Choose $f \in C_c^+(G)$ such that $f(e)>0$ and $f(x)=f(x^{-1})$ for all $x \in G$.
 Put $U =\{ x \in G ;f(x) > 0 \}$, by choosing $U_N=U\cap N$ and according to Lemma \ref{le2.5} there is subset $A$ of $G$ with properties (i) and (ii) which are mentioned in this Lemma.\\
Let for every $y \in A$,  $f^y (x) =f(x y^{-1})$ for   $x \in G$.  By using Proposition \ref{pr2.2}, we can define a continuous rho-function $\rho_{f^y}$ by
\begin{equation*}
\rho_{f^y} (x)=\int_K \int_H \frac{\Delta_G (h)}{\Delta_H (h) \Delta_K (k^{-1})} f(k^{-1} x h y^{-1}) d h d k.
\end{equation*}
Now, by using the fact that $\rho_{f^y}(x)=0$ if $x \not\in KU_NyH$ and applying the Proposition \ref{pr2.2}, for any compact subset $M$ of $G$, we have $\rho_{f^y}$ as being zero on $M$ for all but finitely many $y\in A$.  Thus $\rho = \sum_{y \in A} \rho_{f^y}$ is a continuous function on $G$. Also, it is evident that $\rho$ is a rho-function.\\
According to Lemma \ref{le2.5} (i), for each $x \in G$, there is $y\in A$ such that   $f^y ( k x h ) > 0$ for some $k \in K$ and $h \in H$. Therefore, $\rho_{f^y} (x) > 0$ and hence $\rho (x) > 0$.
\end{proof}
%

 
 Next, we use Proposition \ref{pr2.2} to construct a positive measure on $K \backslash G / H $.
\begin{theorem}\label{thkh3.10}
Let $\rho$ be a rho-function for the triple  $(K,G,H)$. Then there exists a positive Radon  measure $\mu_\rho$ on $K \backslash G / H $ such that 
\[
\int_{K \backslash G / H } Q(f) (\ddot{x})d\mu_\rho(\ddot{x}) =\int_G f(x) \rho(x) dx
\]
for all $f\in C_c(G)$. 
\end{theorem}
\begin{proof}
By applying Proposition \ref{le2.6}, for each  triple $(K, G, H)$ we can get a rho-function $\rho$.
 Take the linear functional $I_\rho$ on $C_c(K \backslash G / H)$ by 
\[
I_\rho\big(Q(f)\big)=\int_G f(x) \rho(x) dx. 
\]
By using Lemma \ref{lekh2.3}, there exists $g\in C_c(G)$ such that $Q(g)(\ddot{x})=1$ on $supp\big(Q(f)\big)$. That is, $\int_K\int_H g(k^{-1}xh))dhdk=1$ for all $x \in Supp \ f$
therefore we can write 
\begin{align*}
\int_G f(x)\rho(x)dx&=\int_G f(x) \rho(x)Q(g)(\ddot{x})dx\\
&=\int_K\int_H\int_G f(x)\rho(x) g(k^{-1}xh) dxdhdk\\
&=\int_G \int_K\int_H f(kxh^{-1}) \Delta _H(h^{-1})\Delta_K(k) \rho(x)  g(x) dh dk dx \\
&=\int_Gg(x)\rho(x)\Big(\int_K\int_H f(k^{-1}xh)dhdk\Big) dx.
\end{align*}
Now if $Q(f)=0$, then $\int_G f(x)\rho(x)dx=0$. Therefore, $I_\rho$ is a well-defined positive linear functional on $C_c(K \backslash G / H )$. We conclude that there exists a positive Radon measure $\mu_\rho$ on $K \backslash G / H $
 such that 
\[
\int_{K \backslash G / H} Q(f) (\ddot{x})d\mu_\rho(\ddot{x})=\int_G f(x) \rho(x)dx.
\]
\end{proof}
We add the $IN$-group condition for closed subgroup $K$ of $G$ in Theorem \ref{thkh3.10} to achive our result.
\begin{theorem}\label{thkh3.9}
Suppose also that $K$ is an $IN$-group. Given any rho-function $\rho$ for the triple $(K, G,H)$, there is an $N$-strongly quasi-invariant measure $\mu_\rho$ on $K \backslash G / H$ such that
\begin{align*}
\int_{G} f(y)\rho(y) dy&= \int_{K\backslash G/H} Q(f)(\ddot{y}) d\mu_{\rho}(\ddot{y})
\\
&= \int_{K\backslash G/H} \int_{K} \int_{H}f(k^{-1}yh)dhdkd \mu_\rho(\ddot{y}).
\end{align*}
\end{theorem}
\begin{proof}
By applying  Theorem \ref{thkh3.10}, we can get a unique measure $ \mu_{\rho} $ on $ K \backslash G / H  $, which satisfies the following:
\begin{align*}
\int_{G} f(y)\rho(y) dy&=  \int_{K\backslash G/H} \int_{K} \int_{H}f(k^{-1}yh)dhdkd\mu_\rho(\ddot{y}).
\end{align*}
$ \mu_{\rho} $ is  an $N$-strongly quasi invariant. Indeed, let
\[
\lambda: N\times K\backslash G/H \times\longrightarrow (0,+\infty) 
\]
by 
\begin{equation}
 \lambda (n,\ddot{y} )= \dfrac{\rho(ny)}{\rho(y)}.\label{eqkh2.4}
\end{equation}
By using the fact that $K$ is an $IN$-group, one can prove  that $\lambda$ is well-defined. 

The continuity of rho-function  $\rho$ results in the fact that $\lambda$ is also continuous. Moreover, for each $n  \in N$,  we have
\begin{align*}
\int_{K\backslash G/H} Q(f)(\ddot{y}) d\mu_{n}(\ddot{y})
&=\int_{K \backslash G / H} Q(L_nf)(\ddot{y})d \mu (y)\\
&=\int_{G} L_n f(y).\rho(y)dy
\\
&=\int_{K \backslash G / H} Q\big(f. \lambda (n,.)\big) (\ddot{y})d \mu_\rho(\ddot{y}).
\end{align*}
Therefore,
\[
 \frac{d\mu_n (\ddot{y})}{d \mu (\ddot{y})}
=\lambda(n,\ddot{y}).
\]
\end{proof}
\begin{remark}
According to Theorem \ref{thkh3.9},
if $K$  is also normal in $G$, then $K \backslash G / H$ has a $G$- strongly quasi-invariant measure.
\end{remark}
In the following proposition we list some properties of $\mu_n$. 
\begin{proposition}
Let $\rho$ be a rho-function for the triple  $(K,G,H)$.
\begin{itemize}
\item[(i).]
If $A$ is a closed subset of $K \backslash G / H $ such that $\rho(n)=0$, for all $n \in N \setminus q^{-1}(A)$, then $Supp \mu_\rho \subseteq A$. 
\item[(ii).]
For each $n \in N$, $L_{n}\rho$ is also a rho-function for the triple $(K,G,H)$ and  $\mu_{L_{n}}\rho=(\mu_\rho)_{n^{-1}}$
\item[(iii).]
Suppose that $K$ is an $IN$-group,  then
if $f \in C_c^+(G)$ and take $\rho=\rho_f$, therefore for any $\alpha\in C_c(K \backslash G / H )$
\[
\int_{K \backslash G / H } \alpha(\ddot{x})d\mu_\rho (\ddot{x})=\int_G \alpha\big(q(x)\big) f(x)dx.
\]
\end{itemize}
\end{proposition}
\begin{proof}
The proof of (i) and (ii) are straightforward.
For each $\alpha\in C_c(K \backslash G / H)$, there is $\varphi \in C_c(G)$ such that $Q(\varphi)=\alpha$. Therefore, we have
\begin{align*}
\int_{K \backslash G / H} \alpha(\ddot{x})d\mu_\rho(\ddot{x})&=\int_{K \backslash G / H} Q(\varphi) (\ddot{x}) d\mu_\rho(\ddot{x})\\
&=\int_G\varphi (x) \rho(x) dx\\
&=\int_G \varphi(x) \int_{K}\int_H \frac{\Delta_G(h)}{\Delta_H(h)}f(k^{-1}xh)dh dkdx\\
&= \int_G \int_{H\times K} \varphi(k^{-1}xh) f(x) d(h\times k)dx \\
&=\int_G f(x) \Big( \int_{H \times K} \varphi(k^{-1} x h) d(h \times k)\Big) dx\\
&=\int_G f(x) Q(\varphi) \big(q(x)\big)dx.
\end{align*} 
This proves (iii). 
\end{proof}
\def \w{\omega}
\def \x{\xi}
\section{rho-function and $N$-strongly quasi-invariant measure}\label{sekh4}
Suppose  that $G$ is a locally compact group, $H$ and $K$ are closed subgroups of $G$ and $N$ is the normalizer group of $K$ in $G$. 
Also, suppose that  $\w$ is a left Haar measure on $N$ with the modular function $\Delta_N$.
In this section, we want to consider under which conditions an  $N$-strongly quasi-invariant measure on $K \backslash G /H$ arises from a rho-function. \\
First, we recall that  if $X$ is a locally compact Hausdorf space and $\mu$ is a positive Radon measure on $X$, then subset $B$ is called locally negligible, if for each compact subset $M$ of $X$, $\mu(B \cap M)=0$. 
\begin{remark}\label{re4.1}

In \cite{khf8} has been shown that $N$ is not locally negligible if and only if $N$ is open subgroup of $G$. 
\end{remark}
\begin{lemma}\label{le4.2}
If $N$ is an  open subgroup of $G$ then each $f\in C_c(N)$ may be regarded as a function in $C_c(G)$ and $Q: C_c(G) \big|_{C_c(N)}$ is surjective on $B=\{F \in C_c(K \backslash G /H),\ supp \ F \subseteq q(N)\}$. 
\end{lemma}
\begin{proof}
Proof is straightforward.
\end{proof}
Our main result  in this section is as follows:
\begin{theorem}\label{th4.4}
Suppose also that  $K$ is an $I N$-group, $N$ is not locally negligible, and $H \subseteq N$. Then every  $N$-strongly quasi-invariant measure $\mu$ on $K \backslash G /H$ arises from a rho-function. That is,   there is  a rho-function $\rho:G \to (0,+\infty)$, such that 
\begin{equation}\label{eqkhllll}
\int_{K \backslash G /H} \int_K \int_H f (k^{-1} xh) dh dk d\mu(\ddot{x}) =\int_G f(x) \rho (x) d x\  \text{for all}\  f\in C_c(N),
\end{equation}
and all such measures are $N$-strongly equivalent. That is to say that they have the same negligible  sets on $q(N)$. 
\end{theorem}
\begin{proof}
Suppose that $\mu$ is an $N$-strongly quasi-invariant measure on $K \backslash G / H $,   then  there is a positive continuous function $\lambda$  on $N \times (K \backslash G / H )$, such that $(d\mu_x / d\mu)(\ddot{y})=\lambda(x,\ddot{y})$.
It is easy to check that
\[
\lambda(n_1n_2,p)=\lambda(n_1,n_2p)\lambda(n_2,p).
\]
According to  Remark \ref{re4.1}, $N$ is an open  subgroup of $G$. Therefore, by applying  Lemma \ref{le4.2}, each function in $C_c(N)$ may be regarded  as a function in $C_c(G)$. Also, Range $  Q\big|_{C_c(N)}$ is $\{F \in C_c(K \backslash G /H); \ supp \ F \subseteq q(N)\}$. The mapping $f \mapsto \int_{K \backslash G /H}Q(f)(KnH) \lambda (n, KH)^{-1}d\mu(KnH)$ is a left invariant positive linear functional on $C_c(N)$. Indeed;
\begin{align*}
&\int_{K \backslash G /H} Q(L_mf) (KnH)\lambda (n,KH)^{-1}d\mu(KnH)\\
&\quad = \int_{K \backslash G /H} L_m Q(f) (KnH) \lambda (n, KH)^{-1}d\mu (KnH) \\
&\quad =\int_{K \backslash G /H} Q(f) (Km^{-1}n H) \lambda (n, KH)^{-1}d\mu(KnH) \\
&\quad =\int_{K \backslash G /H} Q(f) (KnH) \lambda (mn, KH) ^{-1} d\mu_m(KnH)\\
&\quad =\int_{K \backslash G /H} Q(f) (KnH) \lambda (m, KnH)^{-1}\lambda (n, KH)^{-1}\lambda (m, KnH)d\mu (KnH) \\
&\quad =\int_{K \backslash G /H} Q(f) (KnH) \lambda (n, KH)^{-1} d\mu(KnH).  
\end{align*}
By uniqueness of Haar measure on $N$,  there is  $c>0$ such that 
\begin{equation}\label{eqkh222}
\int_{K \backslash G /H} Q(f)(KnH) \lambda (n, KH)^{-1} d\mu (Kn H) =c \int_N f(x) d \w(x). 
\end{equation}
Let $\rho_1:N \to (0,+\infty)$  be given by $\rho_1(n)=c\lambda(n, KH)$. By replacing $f$ by $f\cdot \lambda (n, KH)$ in \eqref{eqkh222}, we see that
\begin{align}
\int_{K \backslash G /H}\int_K\int_H f(k^{-1}nh)dhdk
&=\int_N f(n) \rho_1(n) dn\label{eq2.*}
\end{align}
Now $\rho_1$ can be  extended  on $G$
 by the following definition:
 \begin{align*}
 & \rho :G \to (0,+\infty)\\
&
\rho(x)=\begin{cases}
\rho_1(x) & x \in N \\
0& x \not\in N 
\end{cases},
\end{align*}
then $\rho$ is  a positive continuous function    on $G$. Moreover, if $h_0\in H$ and $k_0 \in K$, then by using $H \subseteq N$, we can write
\begin{align*}
\int_G f(x) \rho (k_0x h_0)dx &= \int_N f(x) \rho(k_0x h_0)d\w(x) \\
&=\int_G f(k_0^{-1}x h_0^{-1})  \rho (x)  \Delta _N(h_0^{-1})d\w(x) \\
&=\int_{K \backslash G /H}\int_H\int_K f(k_0^{-1}k^{-1} x hh_0^{-1} ) dk dh d\mu(\ddot{x}) \\
&= \Delta_G(h_0^{-1})\int_{K \backslash G /H}\int_K\int_H f (k^{-1}x h)\Delta_H (h_0)\Delta_K(k_0) dh dk d\mu(\ddot{x})\\
&=\Delta_G(h_0^{-1})\Delta_H(h_0) \Delta_K(k_0)\int_G f (x)\rho (x) d\w(x) \\
&=\frac{\Delta_K(k_0)\Delta_H(h_0)}{\Delta_G(h_0)}\int _N f(x)\rho(x) d\omega(x).
\end{align*}
This being for all $f\in C_c(N)$,  $\rho(k_0 n h_0)=\frac{\Delta_H(h_0)\Delta_K(k_0)}{\Delta_G(h_0)}\rho(n)$. 
When $x \not\in N$, the equality $\rho(k_0 x h_0)=\frac{\Delta_H(h_0)\Delta_K(k_0)}{\Delta_G (h_0)}\rho(x)$ is trivial. This proves that $\rho$ is a rho-function.\\
Suppose that $\mu_1$ and $ \mu_2$ are $N$-strongly quasi-invariant measures on $K \backslash G /H$  associated with rho-function $\rho_1$ and $\rho_2$ on $G$,  respectively. Then, we have 
\[
\frac{\rho_1(knh)}{\rho_2(knh)}=\frac{\frac{\Delta_H(h)\Delta_K(k)}{\Delta_G(h)}\cdot \rho_1(n)}{\frac{\Delta_H(h)\Delta_K(k)}{\Delta_G(h)} \cdot \rho_2(n)}=\frac{\rho_1(n)}{\rho_2(n)}
\]
for all $n \in N$. \\
Take $\varphi: K \backslash G /H \to [0, +\infty)$ by 
\[
\varphi(KxH)=\begin{cases}
\frac{\rho_1(x)}{\rho_2(x)} & \text{if}\ x \in N \\
0 & \text{if} \ x \not\in N
\end{cases},
\]
clearly
$\varphi$ is well-defined and continuous. \\
Let $f\in C_c(G)$. Then we can write
\begin{align*}
Q\big(f \cdot \frac{\rho_1}{\rho_2}\big) (KnH)&=\int_K\int_H f(k^{-1}n h) \frac{\rho_1(k^{-1}nh)}{\rho_2(k^{-1}nh)}dh dk \\
&= \frac{\rho_1(n)}{\rho_2(n)} \int_K \int_H f(k^{-1} n h) dh dk, 
\end{align*}
for all $n \in N$. \\
Therefore,  we have
\begin{align*}
\int_{K \backslash G /H} Q(f) (\ddot{n}) d\mu_1 (\ddot{n})&= \int_N f(n) \rho_1 (n)dn\\
&=\int_N f(n) \big(\frac{\rho_1(n)}{\rho_2(n)}\big) \rho_2 (n) dn \\
&=\int_{K \backslash G /H} Q(f) (\ddot{n})\varphi(\ddot{n})d\mu_2(\ddot{n}).
\end{align*}
Hence, $\frac{d\mu_1}{d\mu_2}(\ddot{n})=\varphi(\ddot{n})$ for all $\ddot{n} \in K \backslash G /H$.\\
Now, if $A\subseteq q(N)$ is a negligible set with respect to $\mu_1$, then we have 
\[
0=\int_{K \backslash G /H} 1_A (\ddot{x}) d\mu_1(\ddot{x})=\int_{K \backslash G /H} 1_A (\ddot{x}) \varphi(\ddot{x}) d\mu_2 (\ddot{x}).
\]
Therefore, $\int_{K \backslash G /H} 1_A(\ddot{x})\varphi(\ddot{x}) d\mu_2(\ddot{x})=0$. By using the fact that for each $n \in N$ we can get $\varphi(\ddot{n})>0$ we have $\int_{K \backslash G /H}1_A (\ddot{x}) d\mu_2(\ddot{x})=0$, so 
$
\mu_2(A)=0
$.\\
Therefore, the negligible sets of $q(N)$ with respect to $\mu_1$ are the same as the negligible sets,  with respect to $\mu_2$, and we are done.
\end{proof}
\begin{corollary}
Let   $G$ be a semidirect product of $K$ and $ H$, respectively.  Double coset space $K \backslash G /H$ possesses a strongly quasi-invariant measure.
\end{corollary}
\begin{corollary}\label{co4.6}
If $K \vartriangleleft G$,  then each strongly quasi-invariant measure on $K \backslash G /H$ arises from a rho-function. In other words, there exists a rho-function $\rho$ on $G$ such that 
\[
\int_{K \backslash G /H}\int_K\int_H f(k^{-1} x h) dh dk d\mu(\ddot{x}) =\int_G f(x) \rho(x) dx\quad \text{for all} \ f\in C_c(G)
\] 
\end{corollary}
\begin{proof}
It is sufficient to apply Theorem \ref{th4.4} and to note the fact that $N=G$
\end{proof}
\begin{proposition}\label{pr4.8}
If $\mu$ is an  $N$-strongly quasi-invariant measure on $K \backslash G /H$ which arises from a rho-function, then $supp\ \mu=K \backslash G /H$.
\end{proposition}
\begin{proof}
Suppose that $\mu$ is  an $N$-strongly quasi-invariant measure on $K \backslash G /H$ which arises from  a rho-function $\rho$. Therefore, we can write 
\[
\int_{K \backslash G /H}\int_K \int_H f(k^{-1}xh) dhdkd\mu(\ddot{x})=\int_G f(x) \rho(x) dx \quad \text{for all}\ f\in C_c(G).
\]
Now if $supp\ \mu\neq K \backslash G /H$, then there is a non-empty open subset $U$ of $K \backslash G /H$ such that $\mu(U)=0$. By applying Urysohn's Lemma,  there is a non-zero $F \in C_c(K \backslash G /H)$ such that $supp \ F \subseteq U$. Also, there is  a non-zero $f\in C_c(G)$ such that $Q(f)=F$.  So,
\begin{align*}
0&=\int_{K \backslash G /H} F(\ddot{x})d\mu(\ddot{x}) =\int_{K \backslash G /H} Q(f)(\ddot{x})d\mu(\ddot{x}) \\
&=\int_{K \backslash G /H} \int_K \int_H f(k^{-1} x h) dh dk d\mu(\ddot{x})\\
&=\int_G f(x) \rho(x) dx>0
\end{align*}
which is a contradiction. Therefore, $supp\ \mu =K \backslash G /H$. 
 \end{proof}
\begin{proposition}
Let $N$ be open in $G$, and $\mu$ be any $N$-strongly quasi-invariant measure on $K \backslash G /H$ which arises from  a rho-function. Then, for a Borel subset $A\subseteq K \backslash G /H$, $A\cap q(N)$  is locally negligible  if and only if $q^{-1}(A) \cap N$ is locally negligible in $G$. 
\end{proposition}
\begin{proof}
Let $A\subseteq K \backslash G /H$ be  a Borel set such that $A \cap q(N)$ be locally negligible. By intersecting $A\cap q(N)$ with an arbitrary compact subset of $K \backslash G /H$, we may assume, without loss of generality, that $A\cap q(N)$ is relatively compact, that is, $\overline{A\cap q(N)}$ is compact. \\
Let $f \in C^+ (G)$ be such that $f \neq 0$. By applying  Fubini's Theorem, we can write 
\begin{equation}
\int_{K \backslash G /H}\int_N f(x) \cdot 1_{A\cap q(N)} (x \cdot \ddot{y}) d\omega(x) d\mu (\ddot{y})=\int_N \int_{K \backslash G /H}f(x) \cdot 1_{A\cap q(N)} (x \cdot \ddot{y})d\mu(\ddot{y})d\omega(x).\label{eqkhsh}
\end{equation}
Suppose that $\mu\big(A\cap q(N)\big)=0$, then $\mu\big(x^{-1}\cdot A\cap q(N)\big)=0$ for all $x \in N$. Thus, the right hand side of \eqref{eqkhsh} is zero,  and so the left hand side. Therefore,  we may state that 
\[
\int _N f(x) 1_{A\cap q(N)} (x \cdot \ddot{y}) d\omega(x)=0 
\]
for almost all $\ddot{y}\in K \backslash G /H$. \\
Let $C$ be any compact subset of $G$ such that $C \cap N\neq \varnothing$ and $U$ a compact unit  neighbourhood in $G$. Select $f \in C_c^+(G)$ so that $f(x) > 1$ for all $x \in CU^{-1}\cap N$. Since $\mu\big(q(U\cap N)\big)>0$, there exists $y\in U\cap N$ such that 
\[
\int_N f(x) \cdot 1_{A\cap q(N)} (x \cdot \ddot{y})d\omega(x)=0.
\]
So,
\begin{align*}
0&=\Delta_N (y) \int_N f(x)\cdot 1_{A\cap q(N)} \big(q(xy)\big) d\omega(x)\\
&=\int_N f(xy^{-1}) 1_{q^{-1}(A)\cap N} (x) d\omega(x) .
\end{align*}
Now, for each $x\in CU^{-1}\cap N$, we have $f(xy^{-1})\geq 1$ which implies that
\[\int_N 1_{q^{-1}(A)\cap N\cap C} (x)d\omega(x)=0.\]
Thus, $q^{-1}(A)\cap N \cap C$ is negligible set for any compact set $C \subseteq G$, that is, $q^{-1}(A)\cap N$ is locally negligible. \\
Conversely, suppose that $q^{-1}(A)\cap N$ is locally negligible. Again, let $f\in C_c^+(N)$, and $\ddot{y}\in q(N)$  be arbitrary and from now fixed, since $q$ is onto, choose $y\in N$ such that $q(y)=\ddot{y}$. Then,  $x \mapsto f(xy^{-1})$ is continuous with compact support. So, 
\begin{align*}
0&=\int_N f(xy^{-1}) \Delta_N (y^{-1}) 1_{q^{-1}(A)\cap N} (x) d\omega(x) \\
&=\int_N f(x) 1_{q^{-1}(A)\cap N} (xy) d\omega(x) \\
&=\int_N f(x) 1_{A \cap q(N)} \big(x \cdot q(y)\big)d\omega(x).
\end{align*}
Then, the left hand side of \eqref{eqkhsh} is zero, therefore the right hand side is zero as well. Hence, for almost all $x \in N$
\begin{align*}
0&=\int_{K \backslash G /H} f(x)\cdot 1_{A\cap q(N)} (x\cdot \ddot{y}) d\mu(\ddot{y})\\
&=f(x)\cdot \mu\big(x^{-1}\cdot A\cap q(N)\big).
\end{align*} 
Since $f\neq 0$, there is $x \in N$ so that $\mu \big(x^{-1} \cdot A\cap q(N)\big)=0$ which implies that $\mu\big(A\cap q(N)\big)=0$. 
\end{proof}
\begin{theorem}\label{th4.5}
If $K$ is also an $I N$-group and $\mu$ is an  $N$-strongly quasi-invariant measure on $K \backslash G /H$, then $\tilde{\mu}$ defined by $\tilde{\mu}(f)=\int_{K \backslash G /H}Q(f) (\ddot{x}d\mu(\ddot{x})$ has the following property:
\begin{equation}
\int_G f(n x h^{-1}) d\tilde{\mu}(x) =\Delta_H(h) \int_G f(x) \cdot \lambda \big(n, q(x)\big) d\tilde{\mu}(x). \label{eq4.1}
\end{equation}
\end{theorem}
\begin{proof}
Suppose that $\mu$ is an  $N$-strongly quasi-invariant measure. Therefore, there is the continuous positive function $\lambda$ on  $N \times K \backslash G /H$  such that $d\mu_n(\ddot{x})=\lambda(n, \ddot{x})d\mu(\ddot{x})$ for all $n \in N$. 
Hence, by applying Theorem \ref{th2.4}, we have
\begin{align*}
\int_G f(nx h^{-1})d\tilde{\mu}(x)&= \int_G L_{n^{-1}} \circ R_{h^{-1}} f(x) d\tilde{\mu}(x) \\
&=\Delta_H(h) \int_G L_{n^{-1}}f(x) d\tilde{\mu} (x) \\
&= \Delta_H (h) \int_{K \backslash G /H} L_{n^{-1}} Q(f) (\ddot{x} )d\mu(\ddot{x}) \\
&=\Delta_H (h) \int_{K \backslash G /H} Q(f) (\ddot{x}) \lambda (n, \ddot{x})d\mu(\ddot{x})\\
&=\Delta_H(h)\int_G \Big(f \cdot \lambda \big(n, q(\cdot)\big)\Big)(x)d\tilde{\mu}(x) 
\end{align*} 
\end{proof}
\begin{remark}\label{re4.7}
Note that if $K=\{e\}$, then we conclude that each strongly quasi-invariant measure on $G/H$ arises from a rho-function  and if $H=\{e\}$, then  $K \backslash G$ (the right cosets of $K$ in $G$) has $N$-strongly quasi-invariant measure by the left action and if $N$ is not locally negligible, this measure arises from  a rho-function.  
\end{remark}
\begin{remark}
Take $K=H$. Now if $N$ is not locally negligible,  each $N$-strongly quasi-invariant measure on $G//H$ arises from a rho-function. 
\end{remark}

\bibliographystyle{amsplain}

\end{document}